\documentclass[a4paper,11pt]{amsart}

\usepackage[utf8]{inputenc}
\usepackage[all,cmtip]{xy}
\usepackage[english]{babel}
\usepackage{mathtools}
\usepackage{graphicx}
\usepackage{multicol}

\usepackage{amsmath,amssymb,amsthm,amsrefs,mathrsfs}
\usepackage{xfrac}
\usepackage{pifont}

\usepackage{array,tabularx}
\usepackage{xcolor}

\usepackage{tikz,tikz-cd}
\usetikzlibrary{positioning, shapes.geometric}
\usetikzlibrary{arrows,arrows.meta}
\usetikzlibrary{calc}
\tikzcdset{arrow style=tikz, diagrams={>={Computer Modern Rightarrow[width=5.5pt,length=3pt]}}}

\usepackage{enumitem}
\usepackage{cleveref}

\usepackage[retainorgcmds]{IEEEtrantools}

\theoremstyle{plain}
\newtheorem{thm}{Theorem}[section] 
\newtheorem*{thm*}{Theorem}
\newtheorem*{mainthm}{Main Theorem}
\newtheorem{prop}[thm]{Proposition}

\newtheorem{lem}[thm]{Lemma}
\newtheorem*{question*}{Question}

\theoremstyle{definition}

\newtheorem{exa}[thm]{Example}
\newtheorem{rem}[thm]{Remark}

\newcommand*{\myproofname}{Sketch proof}

\newcommand{\Z}{\mathbb{Z}}
\newcommand{\Q}{\mathbb{Q}}

\usepackage{xparse}
\makeatletter
\RenewDocumentCommand{\title}{om}{%
  \IfNoValueTF{#1}
     {\gdef\shorttitle{Homotopy Fibrations Over Connected Sums}}
     {\gdef\shorttitle{#1}}%
  \gdef\@title{#2}%
}
\makeatother

\usepackage{libertine}
\usepackage[libertine]{newtxmath}
\renewcommand{\mathbb}{\varmathbb}

\begin{document}

\title{The Rational Homotopy Type of Homotopy Fibrations Over Connected Sums}
\author{Sebastian Chenery}
\address{Mathematical Sciences, University of Southampton, Southampton SO17 1BJ, United Kingdom}
\email{s.d.chenery@soton.ac.uk}




\begin{abstract}
We provide a simple condition on rational cohomology for the total space of a pullback fibration over a connected sum to have the rational homotopy type of a connected sum, after looping. This takes inspiration from recent work of Jeffrey and Selick, in which they study pullback fibrations of this type, but under stronger hypotheses compared to our result.

\end{abstract}

\maketitle

\section{Introduction}

Taking inspiration from \cite{js}, we begin with a homotopy fibration $F\rightarrow L\xrightarrow{f} C$ in which all spaces have the homotopy type of Poincar\'e Duality complexes; that is to say simply connected, finite dimensional $CW$-complexes whose cohomology rings satisfy Poincar\'e Duality. Writing $dim(C)=n$ and $dim(L)=m$, let $B$ be another $n$-dimensional Poincar\'e Duality complex. Form the connected sum $B\#C$, and take the natural collapsing map $p:B\#C\rightarrow C$. Defining the $m$-dimensional complex $M$ as the pullback of $f$ across $p$, we have a homotopy fibration diagram
\begin{equation}\label{dgm:main1}
    \begin{tikzcd}[row sep=1.5em, column sep = 1.5em]
        F \arrow[rr] \arrow[dd, equal] && M \arrow[dd] \arrow[rr] && B\#C \arrow[dd, "p"] \\
        \\
        F \arrow[rr] && L \arrow[rr, "f"] && C.
    \end{tikzcd}
\end{equation}
A natural question follows: to what extent does $M$ behave like a connected sum?

Jeffrey and Selick give a partial answer to the above question in \cite{js}. They consider the question when each space is a closed, oriented, smooth, simply connected manifold, but in the stricter setting of fibre bundles, and construct a space $X'$ with the property that there is an isomorphism of homology groups \[H_k(M;\Z)\cong H_k(X';\Z)\oplus H_k(L;\Z)\] for $0<k<m$ \cite{js}*{Theorem 3.3}. This suggests that in certain circumstances we might expect there to be an $m$-dimensional manifold $X$, such that $M\simeq X\#L$. Jeffrey and Selick show that there are contexts in which such an $X$ exists, and others where it cannot exist\footnote{At time of writing, it is known that the current arXiv version of \cite{js} contains a mistake. It has been communicated to me privately that a new version has been prepared, which recovers the same main results, and will be published and uploaded online soon.}.

Similar questions to the above have been asked recently. Duan \cite{duan} approaches the topic from a much more geometric, surgery theoretic viewpoint. In this work, the principal objects of concern are manifolds which exhibit a regular circle action; namely, a free circle action on an $n$-dimensional closed, oriented, smooth, simply connected manifold whose quotient space is an $(n-1)$-dimensional closed, oriented, smooth, simply connected manifold. Translated into the context of \cite{js}, Duan studies the situation when $F\simeq S^1$. If $L$ is of dimension at least 5, it is shown in \cite{duan} that the total space of the pullback fibration is indeed always diffeomorphic to a connected sum. Although the thrust of \cite{duan} is mainly concerned with constructing smooth manifolds that admit regular circle actions, it is interesting to remark that its strategy yields a specific class of examples for the situation as in Diagram (\ref{dgm:main1}). Other recent work includes that of Huang and Theriault \cite{huangtheriault}, in which they consider the loop space homotopy type of manifolds after stabilisation by connected sum with a projective space. They do so by combining the results of \cite{duan} with a homotopy theoretic analysis of special cases of Diagram (\ref{dgm:main1}).

In this paper we give a special circumstance, recorded in Proposition \ref{prop:1}, in which the based loop space of $M$ is homotopy equivalent to the based loops of a connected sum. This takes its most dramatic form in the context of rational homotopy theory, which we state in the Main Theorem below. Let $\overline{C}$ and $\overline{L}$ denote the $(n-1)$- and $(m-1)$-skeleta of $C$ and $L$, respectively. 
\begin{mainthm}[Theorem \ref{thm:connsum}]
Given spaces and maps as in Diagram (\ref{dgm:main1}), if 
\begin{enumerate}
    \item[(i)] the fibre map \(F\rightarrow M\) is (rationally) null homotopic, and,
    \item[(ii)] both $H^*(\overline{C};\Q)$ and $H^*(\overline{L};\Q)$ are generated by more than one element,
\end{enumerate}
there is a rational homotopy equivalence $\Omega M\simeq \Omega (X\#L)$ for an appropriate \(CW\)-complex \(X\) which we construct in Section \ref{sec:pullbacks}. 
\end{mainthm}
\noindent Thus we are able to give an affirmative answer in this situation, but after looping and up to rational homotopy equivalence. Examples of homotopy fibrations that fulfil the criteria of the Main Theorem include certain sphere bundles. Furthermore, note that a consequence of the Main Theorem is that there is an isomorphism of rational homotopy groups: $\pi_*(M)\otimes\Q\simeq\pi_*(X\#L)\otimes\Q$.

The author would like to thank their PhD supervisor, Stephen Theriault, for the many enlightening discussions during the preparation of this work, as well as Paul Selick and Lisa Jeffrey for their generosity. The author also wishes to thank the reviewer for their helpful and insightful comments. 

\subsection*{Competing Interests Declaration:} the author declares none.

\section{Preliminaries}

For two path connected and based spaces $X$ and $Y$, the \textit{(left) half-smash} of $X$ and $Y$ is the quotient space \[X\ltimes Y=(X\times Y)/(X\times y_0)\] where $y_0$ denotes the basepoint of $Y$. Furthermore, it is a well known result that if $Y$ is a co-$H$-space, then there is a homotopy equivalence $X\ltimes Y\simeq(X\wedge Y)\vee Y$.

We now move to another definition. For a homotopy cofibration $A\xrightarrow{f} B \xrightarrow{j} C$, the map $f$ is called \textit{inert} if $\Omega j$ has a right homotopy inverse. This an integral version of a notion used in rational homotopy theory, namely \textit{rational inertness}, which we define in Section \ref{sec:rational}. We will make use of the following result, due to Theriault \cite{t20}. 

\begin{thm}[Theriault]\label{thm:splitfib}
Let $A\xrightarrow{f} B \xrightarrow{j} C$ be a homotopy cofibration of simply connected spaces, where the map $f$ is inert. Then there is a homotopy fibration \[\Omega C\ltimes A\rightarrow B\xrightarrow{j}C.\] Moreover, this homotopy fibration splits after looping, so there is a homotopy equivalence $\Omega B\simeq\Omega C\times\Omega(\Omega C\ltimes A).$ \hfill $\square$
\end{thm}


Take now a different situation, in which we have two homotopy cofibrations of simply connected spaces \[A\xrightarrow{f} B \xrightarrow{j} C\text{\; and \;}Y\xrightarrow{i}B\xrightarrow{p}X.\] In the diagram below, each complete row and column is a homotopy cofibration, and the bottom-right square is a homotopy pushout, defining the new space $Q$ and the maps $h$ and $q$:
\begin{equation}\label{dgm:setup}
    \begin{tikzcd}[row sep=1.5em, column sep = 1.5em]
        && Y \arrow[rr, equal] \arrow[dd, "i"] && Y \arrow[dd, "j\circ i"] \\
        &&&&& \\
        A \arrow[rr, "f"] \arrow[dd, equal] && B \arrow[rr, "j"] \arrow[dd, "p"] && C \arrow[dd, "h"] \\
        &&&&& \\
        A \arrow[rr, "p\circ f"] && X \arrow[rr, "q"] && Q.
    \end{tikzcd}
\end{equation}
We record an elementary fact in the following lemma, for ease of reference.

\begin{lem} \label{lem:loophinv}
Take the setup of Diagram (\ref{dgm:setup}). If the maps $\Omega p$ and $\Omega q$ have right homotopy inverses, then so does $\Omega h$. Moreover, there is a homotopy equivalence \[\Omega C\simeq\Omega Q\times\Omega(\Omega Q\ltimes Y).\]
\end{lem}

\begin{proof}
Let us denote the right homotopy inverses of $\Omega q$ and $\Omega p$ by $s$ and $t$, respectively. Then, by the homotopy commutativity of Diagram (\ref{dgm:setup}), $\Omega h$ has right homotopy inverse given by the composite $\Omega j\circ t \circ s$.

As $Y$ and $C$ are simply connected, the space $Q$ is as well, and the map $j\circ i$ is by definition inert. Hence we may apply Theorem \ref{thm:splitfib} to the right-most column of (\ref{dgm:setup}), obtaining the asserted homotopy equivalence.
\end{proof}

Finally, recall that the spaces considered by Jeffrey and Selick in \cite{js} have the homotopy type of oriented, smooth, closed, simply connected manifolds, and are thus Poincar\'e Duality complexes; that is to say, they have have the homotopy type of simply connected $CW$-complexes whose cohomology rings satisfy Poincar\'e Duality. For such a complex there exists a $CW$ structure having a single top-dimensional cell. For brevity, given a $k$-dimensional Poincar\'e Duality complex $Y$, let $\overline{Y}$ denote its $(k-1)$-skeleton, and note that there exists a homotopy cofibration\[S^{k-1}\xrightarrow{f}\overline{Y}\rightarrow \overline{Y}\cup_f e^k\simeq Y\] where $f$ is the attaching map of the top-cell of $Y$. Furthermore, given two \(k\)-dimensional Poincar\'e Duality complexes \(X\) and \(Y\), whose top-dimensional cells are attached by maps \(f\) and \(g\) (respectively), one forms their connected sum by means of the composite \[f+g:S^{k-1}\xrightarrow{\sigma}S^{k-1}\vee S^{k-1}\xrightarrow{f \vee g}\overline{X}\vee\overline{Y}\] where \(\sigma\) is the usual comultiplication. The homotopy cofibre of \(f+g\) is defined to be \(X\#Y\). In particular, \(\overline{X\#Y}\simeq\overline{X}\vee\overline{Y}\), and there is a homotopy cofibration \[\overline{X}\rightarrow X\#Y\rightarrow Y.\]

\section{Pullbacks over Connected Sums}\label{sec:pullbacks}

The situation we wish to study begins with a homotopy fibration \(F\rightarrow L\xrightarrow{f} C\), in which each space has the homotopy type of a Poincar\'e Duality complex. As in the introduction, let $dim(C)=n$ and $dim(L)=m$, and let $B$ be another $n$-dimensional Poincar\'e Duality complex. We form the connected sum $B\#C$, and take the natural collapsing map $p:B\#C\rightarrow C$. Defining the $m$-dimensional complex $M$ as the pullback of $f$ across $p$, we have a homotopy fibration diagram
\begin{equation}\label{dgm:main2}
    \begin{tikzcd}[row sep=1.5em, column sep = 1.5em]
        F \arrow[rr, "\alpha"] \arrow[dd, equal] && M \arrow[dd, "\pi"] \arrow[rr] && B\#C \arrow[dd, "p"] \\
        \\
        F \arrow[rr] && L \arrow[rr, "f"] && C
    \end{tikzcd}
\end{equation}
where we denote the induced map \(M\rightarrow L\) by \(\pi\) and the fibre map \(F\rightarrow M\) by \(\alpha\). 

\begin{lem}\label{lem:pushout}
With spaces and maps as in Diagram (\ref{dgm:main2}), there is a homotopy pushout square
\begin{equation*}
    \begin{tikzcd}[row sep=3em, column sep = 3em]
        F\times\overline{B} \arrow[r] \arrow[d, "p_1"] & M \arrow[d, "\pi"]  \\
        F \arrow[r] & L.
    \end{tikzcd}
\end{equation*}
Where the map \(p_1\) is projection to the first factor. In particular, if \(\alpha\) is null homotopic, there is a homotopy cofibration \(F\ltimes\overline{B}\rightarrow M\xrightarrow{\pi} L\).
\end{lem}

\begin{proof}
To prove the existence of the asserted homotopy pushout, we will use Mather's Cube Lemma. Indeed, consider the following diagram
\begin{equation}\label{dgm:cube}
    \begin{tikzcd}[row sep=1em, column sep = 1em]
        F\times\overline{B} \arrow[rr, "\beta"] \arrow[dr,swap,"p_1"] \arrow[dd,swap,"p_2"] &&
        M \arrow[dd] \arrow[dr] \\
        & F \arrow[rr, crossing over] && L \arrow[dd, "f"] \\
        \overline{B} \arrow[rr] \arrow[dr] && B\#C \arrow[dr, "p"]\\
        & * \arrow[rr] \arrow[from=uu, crossing over] && C. \\
    \end{tikzcd}
\end{equation}
We must show is that (\ref{dgm:cube}) commutes up to homotopy, that bottom face is a homotopy pushout and that the four vertical faces are homotopy pullbacks. 

The bottom face of (\ref{dgm:cube}) arises from the homotopy cofibration \(\overline{B}\rightarrow B\#C\xrightarrow{p}C\), and so is homotopy pushout. The front face is evidently a homotopy pullback, because it comes from the homotopy fibration we began with, as is the right-hand face of the cube, which is the right-hand sqaure in Diagram (\ref{dgm:main2}). Furthermore, it is an elementary fact that the left-hand face of the cube, together with the projection maps, is also homotopy pullback. 

What remains to show is that the map \(\beta:F\times\overline{B}\rightarrow M\) is chosen such that the diagram commutes up to homotopy and that the rear face is a homotopy pullback. Indeed, as the right-hand face is a homotopy pullback, \(\beta\) is induced by the existence of the composites \(F\times\overline{B}\rightarrow F\rightarrow L\) and \(F\times\overline{B}\rightarrow\overline{B}\rightarrow B\#C\), so the diagram does indeed homotopy commute. One then applies \cite{ark}*{Theorem 6.3.3}, which forces the rear face to be a homotopy pullback. In the special case in which the fibre map \(\alpha\) is null homotopic, we may pinch out a copy of \(F\) in the asserted pushout, giving the square 
\begin{equation*}
    \begin{tikzcd}[row sep=3em, column sep = 3em]
        F\ltimes\overline{B} \arrow[r] \arrow[d] & M \arrow[d, "\pi"]  \\
        * \arrow[r] & L.
    \end{tikzcd}
\end{equation*}
which is equivalent to the stated homotopy cofibration.
\end{proof}


\begin{rem}
Note that, because Diagram (\ref{dgm:main2}) is homotopy commutative, requiring \(\alpha\) to be null homotopic forces the fibre map \(F\rightarrow L\) to have also been null homotopic to begin with.
\end{rem}

We now give the thrust of this section, providing a circumstance in which the based loop space of the Poincar\'e Duality complex $M$ is homotopy equivalent to the based loop space of a connected sum. Let $X'=F\ltimes\overline{B}$ and $X=X'\cup e^m$ (the homotopy class of the attaching map \(S^{m-1}\rightarrow X'\) plays no role in what is to follow, so we suppress it in the definition of \(X\)). 



\begin{prop}\label{prop:1}
Take the situation as in Diagram (\ref{dgm:main1}), and suppose that the map $\Omega p$ has a right homotopy inverse. Then the map $\Omega \pi$ has a right homotopy inverse. Moreover,  if \(\alpha\) is null homotopic and the attaching map of the top cell of $L$ is inert, then \[\Omega M\simeq \Omega(X\#L).\]
\end{prop}

\begin{proof}

Denoting the right homotopy inverse of $\Omega p$ by $s:\Omega C\rightarrow \Omega(B\#C)$, consider the diagram
\begin{equation*}
    \begin{tikzcd}[row sep=1.5em, column sep = 1.5em]
        && \Omega L \arrow[dddr, equal, bend right=20] \arrow[dr, dashed, "\lambda"] \arrow[r, "\Omega f"] & \Omega C \arrow[drr, "s"] &&& \\
        &&& \Omega M \arrow[rr] \arrow[dd, "\Omega\pi"] && \Omega(B\#C) \arrow[dd, "\Omega p"] \\
        &&&&&&& \\
        &&& \Omega L \arrow[rr, "\Omega f"] && \Omega C
    \end{tikzcd}
\end{equation*}
where the map $\lambda$ will be detailed momentarily. Since the right-hand square of Diagram (\ref{dgm:main2}) is a homotopy pullback, so is the square in the above. Furthermore, since $\Omega p\circ s\simeq 1_{\Omega C}$, the diagram commutes. As $\Omega M$ is the homotopy pullback of $\Omega f$ across $\Omega p$, the map $\lambda$ exists and we have that $\Omega\pi\circ\lambda\simeq1_{\Omega L}$. In other words, the map $\lambda$ is a right homotopy inverse for $\Omega\pi$. 

Consequently, in the case when \(\alpha\) is null homotopic, we apply Theorem \ref{thm:splitfib} to the homotopy cofibration \[X'\rightarrow M\xrightarrow{\pi}L\] from the special case of Lemma \ref{lem:pushout}. Indeed, since $\Omega \pi$ has a right homotopy inverse, the map $X'\rightarrow M$ is by definition inert, so Theorem \ref{thm:splitfib} immediately gives us that \[\Omega M\simeq \Omega L\times \Omega(\Omega L\ltimes X').\] On the other hand, let us now consider the connected sum $X\#L$. Take two homotopy cofibrations: one is the attaching map of the top-cell of $X\#L$, and the other is from inclusion of a wedge summand \[S^{m-1}\rightarrow X'\vee \overline{L}\rightarrow X\#L\text{\; and \;}X'\hookrightarrow X'\vee \overline{L}\xrightarrow{q} \overline{L}.\]
We combine these to give a cofibration diagram, in the sense of (\ref{dgm:setup})
\begin{equation*}
    \begin{tikzcd}[row sep=1.5em, column sep = 1.5em]
        &&& X' \arrow[rr, equal] \arrow[dd, hook] && X' \arrow[dd] \\
        &&&&&&& \\
        &S^{m-1} \arrow[rr] \arrow[dd, equal] && X'\vee \overline{L} \arrow[dd, "q"] \arrow[rr] && X\#L \arrow[dd] \\
        &&&&&&& \\
        & S^{m-1} \arrow[rr] && \overline{L} \arrow[rr, "j"] && L. 
    \end{tikzcd}
\end{equation*}
The map $q$ pinches to the second wedge summand, and therefore has a right homotopy inverse given by inclusion; therefore $\Omega q$ also has a right homotopy inverse. Moreover, if the attaching map of the top-cell of $L$ is inert, the map $\Omega j$ has a right homotopy inverse, by definition. Thus Lemma \ref{lem:loophinv} applies, implying there is a homotopy equivalence \[\Omega(X\#L)\simeq \Omega L\times \Omega(\Omega L\ltimes X').\] Thus $\Omega M$ and $\Omega(X\#L)$ are both homotopy equivalent to $\Omega L\times \Omega(\Omega L\ltimes X')$, and are therefore homotopy equivalent to each other.
\end{proof}

\begin{exa}
A general class of examples that satisfy the requirement that \(\alpha\simeq*\) are sphere bundles, \(S^r\rightarrow L\rightarrow C\), where the pullback \(M\) has trivial \(r^{th}\) homotopy group. Consider for example the classical Hopf bundle \(S^1\rightarrow S^3\xrightarrow{\eta} S^2.\) Taking products with the trivial fibration $*\rightarrow S^4\rightarrow S^4$ yields a new homotopy fibration \[S^1\rightarrow S^3\times S^4\xrightarrow{\eta\times1}S^2\times S^4.\] Applying our construction with $B=S^3\times S^3$, we have the following pullback diagram of homotopy fibrations 
\begin{equation*}
    \begin{tikzcd}[row sep=1.5em, column sep = 1.5em]
        & S^1 \arrow[rr] \arrow[dd, equal] && M \arrow[dd, "\pi"] \arrow[r] & (S^3\times S^3)\#(S^2\times S^4) \arrow[dd, "p"]\\
        \\
        & S^1 \arrow[rr] && S^3\times S^4 \arrow[r, "\eta\times1"] & S^2\times S^4.
    \end{tikzcd}
\end{equation*}
Using techniques from \cite{t20}*{Section 9} it can be shown that $\Omega(S^2\times S^4)$ retracts off $\Omega((S^3\times S^3)\#(S^2\times S^4))$ via a right homotopy inverse for $\Omega p$. Moreover, the attaching map for the top-cell of the product $S^3\times S^4$ is known to be inert.

Now we show that \(\pi_1(M)\cong0\). The existence of a right homotopy inverse for the map \(\Omega p\) implies that its homotopy fibre (and consequently the homotopy fibre of \(\pi\)) is homotopy equivalent to \(\Omega(\Omega(S^2\times S^4)\ltimes(S^3\vee S^3))\), by Theorem \ref{thm:splitfib}. It is now easy to check that the long exact sequence of homotopy groups induced by the fibration sequence \(\Omega(\Omega(S^2\times S^4)\ltimes(S^3\vee S^3))\rightarrow M\xrightarrow{\pi}S^3\times S^4\) forces \(\pi_1(M)\) to be trivial. Therefore Proposition \ref{prop:1} applies, with \[X'\simeq S^1\ltimes(S^3\vee S^3)\simeq S^3\vee S^3\vee S^4\vee S^4.\] By gluing a 7-cell to $X'$, we may take $X=(S^3\times S^4)\#(S^3\times S^4)$. Hence we obtain a homotopy equivalence \[\Omega M\simeq \Omega((S^3\times S^4)\#(S^3\times S^4)\#(S^3\times S^4)).\] To conclude this example, we remark that many of the situations considered by Duan in \cite{duan} also fit into this framework.
\end{exa}

\section{The Rational Homotopy Perspective}\label{sec:rational}

We wish to apply Proposition \ref{prop:1} in the context of rational homotopy theory. Let \[S^{k-1}\xrightarrow{f}Y\xrightarrow{i}Y\cup_f e^k\] be a homotopy cofibration, where the map $f$ attaches a $k$-cell to $Y$ and $i$ is the inclusion. The map $f$ is \textit{rationally inert} if $\Omega i$ induces a surjection in rational homology. This implies that, rationally, $\Omega i$ has a right homotopy inverse. The following theorem was first proved in \cite{hl}*{Theorem 5.1}, though we prefer the statement found in \cite{fh}.

\begin{thm}[Halperin-Lemaire]\label{thm:hl}
If $Y\cup_f e^k$ is a Poincar\'e Duality complex and $H^*(Y;\Q)$ is generated by more than one element, then the attaching map $f$ is rationally inert. \hfill $\square$
\end{thm} 

\noindent This leads us to the statement and proof of the Main Theorem.

\begin{thm}\label{thm:connsum}
Given spaces and maps as in Diagram (\ref{dgm:main2}), if 
\begin{enumerate}
    \item[(i)] the map \(\alpha\) is (rationally) null homotopic, and,
    \item[(ii)] both $H^*(\overline{C};\Q)$ and $H^*(\overline{L};\Q)$ are generated by more than one element,
\end{enumerate}
there is a rational homotopy equivalence $\Omega M\simeq \Omega (X\#L)$. 
\end{thm}

\begin{proof}
By Theorem \ref{thm:hl}, the attaching maps for the top-cells of $C$ and $L$ are rationally inert. We have the homotopy pullback below, which is the right-hand square of (\ref{dgm:main2})
\begin{equation*}
    \begin{tikzcd}[row sep=1.5em, column sep = 1.5em]
        M \arrow[rr] \arrow[dd, "\pi"] && B\#C \arrow[dd, "p"] \\
        \\
        L \arrow[rr, "f"] && C.
    \end{tikzcd}
\end{equation*}
Rationalising spaces and maps in this pullback square, we see that Proposition \ref{prop:1} would apply if the map $\Omega p$ has a rational right homotopy inverse, as the attaching map for the top-cell of $L$ is rationally inert. Thus we would have a rational homotopy equivalence $\Omega M\simeq\Omega (X\#L)$. 

It therefore remains to show that the map $\Omega p$ has a rational right homotopy inverse. With this in mind, consider the following homotopy cofibration diagram
\begin{equation*}
    \begin{tikzcd}[row sep=1.5em, column sep = 1.5em]
        &&& \overline{B} \arrow[rr, equal] \arrow[dd, hook] && \overline{B} \arrow[dd] \\
        &&&&&&& \\
        &S^{n-1} \arrow[rr] \arrow[dd, equal] && \overline{B}\vee \overline{C} \arrow[dd, "q"] \arrow[rr] && B\#C \arrow[dd, "p"] \\
        &&&&&&& \\
        & S^{n-1} \arrow[rr] && \overline{C} \arrow[rr, "i_C"] && C. 
    \end{tikzcd}
\end{equation*}
As the pinch map $q$ has a right homotopy inverse, so does $\Omega q$. Furthermore, the attaching map of the top-cell of $C$ is rationally inert, and therefore $\Omega i_C$ has a right homotopy inverse after rationalisation. Therefore the map $\Omega p$ also has a (rational) right homotopy inverse, by Lemma \ref{lem:loophinv}.
\end{proof}

\begin{rem}
Recall that a simply connected space $Y$ is called \textit{rationally elliptic} if $dim(\pi_*(Y)\otimes\Q)<\infty$, and called\textit{ rationally hyperbolic} otherwise \cite{fht}. We remark briefly on the rational hyperbolicity of the spaces discussed above. 

Indeed, suppose that the skeleton $\overline{B}$ is a suspension. Then, as $X'=F\ltimes\overline{B}$, we have a homotopy equivalence $X'\simeq(F\wedge\overline{B})\vee\overline{B}$, which is again a suspension. Thus, rationally, $X'$ is homotopy equivalent to a wedge of spheres. Assuming $X'$ is rationally homotopy equivalent to a wedge containing more than one sphere of dimension greater than 1, this would imply the rational hyperbolicity of $\Omega X$. Indeed, this is guaranteed if the ring $H^*(\overline{B};\Q)$ has more that one generator of degree 2 or more, or if $H^*(\overline{B};\Q)$ has one such generator and $F$ is not rationally contractible. Since $X'$ homotopy retracts off $\Omega L\ltimes X'$, by Theorem \ref{thm:connsum} we have that $\Omega X'$ retracts off $\Omega M$. With the assumptions on $X'$ above, this implies that $\Omega M$ is rationally hyperbolic. 

As a final observation, note that a natural situation in which $\overline{B}$ has the homotopy type of a suspension would be when $B$ is sufficiently highly connected: by \cite{ganeacogroups}, if $B$ is $k$-connected, $\overline{B}$ has the homotopy type of a suspension if $n\leq3k+1$. For example, take $B$ to be an $(n-1)$-connected $2n$-dimensional Poincar\'e Duality complex. 
\end{rem}

\bibliographystyle{amsplain}
\bibliography{bib}

\end{document}